\documentclass[12pt]{amsart}
\usepackage{amsmath}
\usepackage[all]{xy}
\usepackage{amssymb}
\usepackage{amscd}
\newtheorem{thm}{Theorem}[section]
\newtheorem{cor}[thm]{Corollary}
\newtheorem{lem}[thm]{Lemma}
\newtheorem{prop}[thm]{Proposition}
\theoremstyle{definition}
\newtheorem{defn}[thm]{Definition}
\newtheorem{ex}[thm]{Example}
\theoremstyle{remark}
\newtheorem{rem}[thm]{Remark}
\newtheorem{prob}[thm]{Problem}
\numberwithin{equation}{section}

\DeclareMathOperator{\Z}{\mathbb{Z}}

\begin{document}

\author[P.V. Danchev]{Peter V. Danchev}
\address{Institute of Mathematics and Informatics, Bulgarian Academy of Sciences \\ "Acad. G. Bonchev" str., bl. 8, 1113 Sofia, Bulgaria.}
\email{danchev@math.bas.bg; pvdanchev@yahoo.com}

\title[Generalizing Regularity and Strong $\pi$-Regularity] {A Generalization of Regular and Strongly $\pi$-Regular Rings}

\maketitle

\begin{abstract} We introduce and investigate the so-called {\it D-regularly nil clean rings} by showing that these rings are, in fact, a non-trivial generalization of the classical von Neumann regular rings and of the strongly $\pi$-regular rings. Some other close relationships with certain well-known classes of rings such as $\pi$-regular rings, exchange rings, clean rings, nil-clean rings, etc., are also demonstrated. These results somewhat supply a recent publication of the author in Turk. J. Math. (2019) as well as they somewhat expand the important role of the two examples of nil-clean rings obtained by \v{S}ter in Lin. Algebra \& Appl. (2018). Likewise, the obtained symmetrization supports that for exchange rings established by Khurana et al. in Algebras \& Represent. Theory (2015).
\end{abstract}

\section{Introduction and Background}

Throughout this paper, all rings are assumed to be associative and unital with identity element $1$ different to the zero element $0$. Our standard terminology and notations are mainly in agreement with \cite{L}. Exactly, for such a ring $R$, we let $U(R)$ denote the set of all units in $R$, $Id(R)$ the set of all idempotents in $R$, $Nil(R)$ the set of all nilpotents in $R$, $J(R)$ the Jacobson radical of $R$, and $C(R)$ the center of $R$. About the specific notions, we shall provide them in detail in the sequel.

Referring to \cite{G} for more information, we remember that a ring $R$ is called {\it von Neumann regular} or just {\it regular} for short if, for every $a\in R$, there is $b\in R$ such that $a=aba$. If $b\in U(R)$, the ring is said to be {\it unit-regular} (see, for instance, \cite{E} or \cite{T} as well). In particular, if $a=a^2b$ (with $ab=ba$), then $R$ is termed {\it strongly regular}. Generalizing both the notions of regular rings and local rings (that are rings for which the quotient modulo their Jacobson radical is a division ring), it was defined in \cite{N1} the so-called {\it NJ-rings} classifying them in \cite[Theorems 2.4]{N1} as those rings $R$ for which each element not in $J(R)$ is regular.

In that direction, as proper generalizations of the concepts stated above, we recollect also that a ring $R$ is called {\it $\pi$-regular} if, for each $a\in R$, there is $n\in \mathbb{N}$ depending on $a$ and having the property $a^n \in a^nRa^n$. On the same vein, we recall that a ring $R$ is said to be {\it strongly $\pi$-regular} if, for each $a\in R$, there is $n\in \mathbb{N}$ depending on $a$ and possessing the property $a^n\in a^{n+1}R\cap Ra^{n+1}$. It was established in \cite{A} that strongly $\pi$-regular rings are themselves $\pi$-regular, whereas the converse holds for abelian rings. Besides, strongly $\pi$-regular rings are always unit-regular, provided they are regular. Notice that $\pi$-regularity was generalized in two different and non-trivial ways in \cite{DS} and \cite{D2}, respectively.

On the other hand, in \cite{N2} it was introduced and intensively studied the class of {\it clean rings} $R$ as those rings for which $R=U(R)+Id(R)$, that is, for any $r\in R$, there exist $u\in U(R)$ and $e\in Id(R)$ such that $r=u+e$. If, in addition, $ue=eu$, then $R$ is termed {\it strongly clean} (see \cite{N3}). It was proved in \cite[Theorem 1]{CK} (compare also with \cite[Theorem 5]{CY}) that unit-regular rings are always clean, but in \cite[Section 2]{NS} was constructed an ingenious example of a unit-regular ring which is {\it not} strongly clean. Likewise, two fundamental examples due to G.M. Bergman (see, for example, \cite[Examples 1,2; pp.13-14]{H}) are guarantors that there exists a regular ring which is {\it not} clean (see, for instance, \cite[p.4746]{CY}) -- the reader can also consult with \cite[Example 5.12]{G} where there is an example due to Bergman of a unit-regular ring with a regular subring which is not unit-regular. Moreover, as showed in \cite{N3}, strongly $\pi$-regular rings are strongly clean, while it was shown there that there is a strongly clean ring (with non nil Jacobson radical) which is not strongly $\pi$-regular -- to specify, such an example is, in fact, a commutative construction.

By a reason of similarity, a {\it nil-clean ring} $R$ was defined in \cite{D} to be the one for which $R=Nil(R)+Id(R)$, i.e., for any $r\in R$, there are $q\in Nil(R)$ and $e\in Id(R)$ such that $r=q+e$. If, in addition, $qe=eq$, then $R$ is termed {\it strongly nil-clean}. It is quite simple to check that nil-clean rings are always clean, while the converse demonstrably fails. It was shown in \cite{DS}, however, that nil-clean rings of bounded index of nilpotence have to be strongly $\pi$-regular. Nevertheless, in \cite[Example 3.1]{S} was illustrated an example of a nil-clean unit-regular ring which is {\it not} strongly $\pi$-regular, while in \cite[Example 3.2]{S} this construction was refined to the exhibition of a nil-clean ring which is {\it not} $\pi$-regular. To author's knowledge there is no a constructive example of a nil-clean ring that is {\it not} strongly clean. Also, it is principally well known that strongly nil-clean rings are always strongly $\pi$-regular rings (compare with \cite{D}) as well as that strongly nil-clean rings were completely characterized independently in \cite{DL} and \cite{KWZ}. Note that nil-cleanness was nontrivially expanded in \cite{DS} and \cite{D0}, respectively. It is also worthwhile noticing that some stronger versions of cleanness and nil-cleanness were explored in \cite{D1}.

As a rather more general setting, in \cite{W} were introduced by using module theory the so-termed {\it exchange rings}. However, a more suitable for ring theory application is the form of Goodearl-Nicholson (see, for example, \cite{GW} and \cite{N2}) which states that a ring $R$ is exchange if, and only if, for every $r\in R$, there exists $e\in rR\cap Id(R)$ such that $1-e\in (1-r)R$ -- this relation is left-right symmetric in the sense that, for each $r\in R$, there exists $f\in Rr\cap Id(R)$ such that $1-f\in R(1-r)$. This criterion was considerably strengthened in \cite{KLN} to the following double symmetrization: {\it A ring $R$ is exchange if, and only if, for any $r\in R$, there is $e\in rRr\cap Id(R)$ such that $1-e\in (1-r)R(1-r)$}. Thus either $e\in rR\cap Id(R)$ with $1-e\in R(1-r)$, or $e\in Rr\cap Id(R)$ such that $1-e\in (1-r)R$. Exchange rings contain both the classes of $\pi$-regular rings (see, e.g., \cite[Example 2.3]{Sto}) and clean rings (see, e.g., \cite{N2}), and so they definitely substitute a quite large class of rings. However, all abelian exchange rings are clean (see \cite{N2}).

On the other side, as substantial extension of the aforementioned $\pi$-regular rings, it was recently defined in \cite{D2} the class of so-called {\it regularly nil clean rings} as those rings $R$ for which, for any $r\in R$, there exists $e\in Rr\cap Id(R)$ with the property that $r(1-e)\in Nil(R)$ (or, equivalently, $(1-e)r\in Nil(R)$. It was given in \cite[Proposition 1.3]{D2} the left-right symmetric property, namely that there exists $f\in rR$ with the property that $r(1-f)\in Nil(R)$ (or, in an equivalent form, $(1-f)r\in Nil(R)$. It was established also in \cite[Proposition 2.4]{D2} that regularly nil clean rings are themselves exchange and that $\pi$-regular rings are themselves regularly nil clean, thus extending the aforementioned result from \cite{Sto}. Since in the abelian case all regularly nil clean rings were strongly $\pi$-regular, there is a clean ring which is {\it not} regularly nil clean. 

\medskip

So, inspired by the validity of the lastly presented facts, we arrive quite naturally at our basic tool, which actually was originally stated into consideration in \cite[Problem 3.1]{D2} searching for the increasing property of the existing idempotent like this $e\in rRr$.

\begin{defn}\label{1} A ring $R$ is said to be {\it double regularly nil clean} or just {\it D-regularly nil clean} for short if, for each $a\in R$ there exists $e\in (aRa)\cap Id(R)$ such that $a(1-e)\in Nil(R)$ (and hence that $(1-e)a\in Nil(R)$). If, in addition, there is a fixed positive integer $k$ such that $[a(1-e)]^k=0$, then we shall say that the D-regularly nil clean ring has {\it index at most $k$}. (This supplies that $[(1-e)a]^{k+1}=0$.)
\end{defn}

Clearly, D-regularly nil clean rings are always regularly nil clean. Reformulating \cite[Problem 3.1]{D2}, an actual question is of whether or not the properties of being regularly nil clean and D-regularly nil clean are independent of each other, i.e., does there exist a regularly nil clean ring that is {\it not} D-regularly nil clean.

Obvious examples of D-regularly nil clean rings are the strongly regular rings, that are, subdirect products of division rings, as well as the local rings with nil Jacobson radical. Even more, strongly regular rings are D-regularly nil clean with index $1$; it is hopefully that the converse will eventually hold. In fact, $a(1-e)=0$ for some $e=aba$ with $b\in R$ yields that $a=ae=a^2ba\in a^2R$. However, it is not obvious whether (or eventually not) we will have that $a\in Ra^2$. Some other kinds of non-trivial examples of such type of rings will be given below.

\medskip

Our work is structured as follows: In the next section, we state and prove our major results (see, respectively, Theorem~\ref{strpireg} as well as Examples~\ref{cur1} and \ref{cur2} listed below). We finish with some useful commentaries treating the more insightful exploration of the current subject, and we also pose some left-open problems.

\section{Preliminary and Main Results}

We begin here with the following technicality, which could be useful for further applications.

\begin{lem}\label{reduct} Suppose that $R$ is a ring. Then $R$ is D-regularly nil clean if, and only if, $R/J(R)$ is D-regularly nil clean and $J(R)$ is nil.
\end{lem}

\begin{proof} Before beginning to prove the statement, we need the following folklore fact:

\medskip

{\it If $P$ is a ring with a nil-ideal $I$ and if $d\in P$ with $d+I\in Id(P/I)$, then $d+I=e+I$ for some $e\in Id(P)\cap dPd$ such that $de=ed$}.

\medskip

The left-to-right implication being valid in the same manner as in \cite[Theorem 2.9]{D2}, we will deal with the right-to-left one. So, given an arbitrary element $a$ of $R$, there exists $b+J(R)\in Id(R/J(R))\cap (a+J(R))(R/J(R))(a+J(R))$ with $(a+J(R))(1+J(R)-(b+J(R))\in Nil(R/J(R))$. Consequently, bearing in mind the above folklore fact, there is $r\in R$ such that $b+J(R)=(a+J(R))(r+J(R))(a+J(R))=ara+J(R)=e+J(R)$ for some $e\in Id(R)\cap (ara)R(ara)\subseteq Id(R)\cap aRa$. Furthermore,

$$
(a+J(R))(1+J(R)-(e+J(R)))=(a+J(R))(1-e+J(R))=
$$

\medskip

$$
=a(1-e)+J(R)\in Nil(R/J(R))
$$

\medskip

\noindent and, therefore, there exists $m\in \mathbb{N}$ having the property that $[a(1-e)]^m\in J(R)\subseteq Nil(R)$. This means that $a(1-e)\in Nil(R)$, as required.
\end{proof}

Although it has been long known that the center of an exchange ring need not to be again exchange, the following statement is somewhat curious even in the light of \cite[Proposition 2.7]{D2}.

\begin{prop} The center of a D-regularly nil clean ring is again a D-regularly nil clean ring.
\end{prop}

\begin{proof} Letting $R$ be such a ring and given $c\in C(R)$, we may write that $(c(1-e))^m=c^m(1-e)=0$ for some $e\in Id(R)\cap cRc=c^2R$. What suffices to prove is that $e\in C(R)$. To do that, for all $r\in R$, it must be that $er(1-e)\in c^mR(1-e)=Rc^m(1-e)=0$ as $e\in c^2R$ implies at once that $e=e^m\in c^mR$. Thus $er=ere$ and, by a reason of similarity, we also have that $re=ere$. Hence, it is now immediate that $er=re$, proving the claim about the centrality of $e$.

What remains to be shown is just that $e\in cC(R)c=c^2C(R)$. Indeed, write $e=c^2b$ for some $b\in R$. This forces that $e=c^2be=c^2y$, where $y=be=eb$ as $e$ is central. We claim that $y\in C(R)$, as needed. In fact, for any $z\in R$, one derives that $yz(1-e)=(1-e)yz=(1-e)ebz=0$ and that $(1-e)zy=zy(1-e)=zbe(1-e)=0$, because $1-e\in C(R)$, which tells us that $yz=yze$ and $zy=ezy$. Further, $yz=yzc^2y=c^2yzy$ and $zy=c^2yzy$ and, finally, $yz=zy$, as claimed.
\end{proof}

It was proved in \cite{N2} that the corner subring of any exchange ring is again exchange as well as in \cite{D2} that the corner subring of any regularly nil clean ring is again regularly nil clean. The next statement parallels to these two assertions, and is also in sharp contrast with the case of clean rings (see \cite{S0}).

\begin{prop}\label{corner} If $R$ is a D-regularly nil clean ring, then so is the corner ring $eRe$ for any $e\in Id(R)$. In particular, if $\mathbb{M}_n(R)$ is D-regularly nil clean, then so does $R$.
\end{prop}

\begin{proof} Choose an arbitrary element $ere\in eRe$ for some $r\in R$. Since $ere\in R$, it follows that there is an idempotent $f$ in $R$ with $f\in (ere)R(ere)$ such that $(1-f)ere\in Nil(R)$. But this could be written as $ere-fere=(e-f)ere=(e-fe)ere=q\in Nil(R)$. Thus $(e-efe)ere=eq=eqe\in Nil(eRe)$ with $efe\in Id(eRe)\cap (ere)(eRe)(ere)=Id(eRe)\cap (ere)R(ere)$, because $efe=f$ and $qe=q$ so that $eq\in Nil(R)$, as expected.

The second part-half appears to be a direct consequence of the first part-half.
\end{proof}

An important but seemingly rather difficult problem is the reciprocal implication of the last assertion, namely if both $eRe$ and $(1-e)R(1-e)$ are D-regularly nil clean rings, does the same hold for $R$ too? Note that for both exchange rings and clean rings this implication is true (see, e.g., \cite{N2} and \cite{HN}, respectively), although for regular and unit-regular rings this is not so (see \cite{G}).

\medskip

The next two tricker technicalities are pivotal.

\begin{lem}\label{reciproc} If $R$ is a ring and $x,y\in R$ with $x=xyx$, then for the element $y':=yxy$ the following two relations are fulfilled:

\medskip

(*) ~ $x=xy'x$;

\medskip

(**) ~ $y'=y'xy'$.
\end{lem}

\begin{proof} About the first relationship, $xy'x=x(yxy)x=(xyx)yx=xyx$. As for the second one, $y'xy'=(yxy)x(yxy)=y(xyx)yxy=y(xyx)y=yxy=y'$, as promised.
\end{proof}

It is worthwhile noticing that in \cite{D2} it was showed that if $a$ is a $\pi$-regular element, that is, $a^n$ is regular for some $n\in \mathbb{N}$, then $a$ is regularly nil clean, too. Nevertheless, this cannot be happen in the situation of D-regular nil cleanness. Specifically, the following is valid:

\begin{prop}\label{twodeg} If $R$ is a ring having an element $a$ such that $a^n$ is regular for some $n\geq 2$, then $a$ is D-regularly nil clean of index not greater than $n$.
\end{prop}

\begin{proof} Writing $a^n=a^nba^n$ for some existing $b\in R$, then with Lemma~\ref{reciproc} at hand, we may also write that $b=ba^nb$. Indeed, setting $b'=ba^nb$, by consulting with the cited lemma we will have that $a^n=a^nb'a^n$ and that $b'=b'a^nb'$, so that without loss of generality we could replace $b'$ via $b$. Furthermore, letting $e:=aba^{n-1}$, we easily check that $e\in Id(R)\cap (aRa)$. By a direct inspection, which we definitely leave to the reader, one verifies that $[a(1-e)]^n=0$, as expected.
\end{proof}

We are now ready to proceed by proving two of the major assertions motivated the writing of this article.

\begin{thm}\label{reg} Regular rings are D-regularly nil clean of index $2$.
\end{thm}

\begin{proof} In such rings each power of an arbitrary element is always a regular element, so that the claim follows immediately from the crucial Proposition~\ref{twodeg} by putting $n=2$.
\end{proof}

We continue with the second statement of this branch.

\begin{thm}\label{strpireg} Strongly $\pi$-regular rings are D-regularly nil clean.
\end{thm}

\begin{proof} For an arbitrary $a\in R$, where $R$ is strongly $\pi$-regular, we may write in accordance with \cite{A} (see \cite{N3} too) that $a^n=a^{2n}x=a^nxa^n$ for some $n\in \mathbb{N}$ and some $x\in R$ with $xa=ax$. So, $a^nx=e\in Id(R)$ and by squaring we deduce that $e=a^{2n}x^2=a^nx^2a^n\in aRa$. Furthermore, $a(1-e)=a(1-a^nx)\in Nil(R)$ since

$$
[a(1-e)]^n=[a(1-a^nx)]^n=a^n(1-a^nx)^n=a^n(1-a^nx)(1-a^nx)^{n-1}=
$$

$$
= (a^n-a^{2n}x)(1-a^nx)^{n-1}=0,
$$

\medskip

\noindent as expected.
\end{proof}

As a valuable consequence, we yield the following partial answer to \cite[Question 3.17]{KYZ}:

\begin{cor} Let $R$ be a ring and let $n\geq 2$ be a fixed positive integer. If all elements of $R$ satisfy the equation $x^n-x\in Nil(R)$, then $R$ is D-regularly nil clean.
\end{cor}

\begin{proof} It is not too hard to verify that such a ring $R$ is strongly $\pi$-regular, and so Proposition~\ref{strpireg} applies to get the claim.
\end{proof}

Now, we will treat the case of vector spaces which results could be of independent interest and which will be also applied in the examples stated below.

\begin{lem}\label{vs} Let $V$  be a vector space over an arbitrary field $K$, let $R=End_K(V)$ whose elements are being written to the left of elements of $V$, and let $a\in R$. Then there exists an idempotent $e\in aRa$ such that $(a(1-e))^2=0$.
\end{lem}

\begin{proof} Let $V_1=(ker a)\cap (im a)$, let $V_2$ be a complement of $V_1$ in $im a$, let $V_3$ be a complement of  $V_1$ in $ker a$, and let $V_4$ be a complement of $(ker a) \oplus (im a)$ in $V$.

Then, we elementarily see that the decomposition holds:

$$
(1) ~ V=V_1\oplus V_2 \oplus V_3 \oplus V_4.
$$

\medskip

Suppose now $B_1$, $B_2$, $B_3$, $B_4$ are the corresponding $K$-bases of $V_1$, $V_2$, $V_3$, $V_4$, respectively.

We claim that one can find an element $r\in R$ such that

$$
(2) ~ \forall ~ v\in B_2 : arav=v; \forall ~ v\in B_4 : rav=0.
$$

\medskip

\noindent To see this, note that in view of point (1), one has that $V_2\oplus V_4$  is disjoint from $ker a=V_1 \oplus V_3$, hence $a$ is one-to-one on the direct sum $V_2\oplus V_4$, whence it maps linearly independent elements of $V_2\oplus V_4$ to linearly independent elements of $V$. Therefore, one can choose an element of the ring $R$ which maps the elements $av$ for $v\in B_2\cup B_4$ to arbitrarily chosen elements of the space $V$. Now, since $V_2$ lies in $im a$, every $v\in V_2$ has a pre-image under the action of $a$ in $V$, and we thereby can choose $r$ to map each
element $av$ ($v\in V_2$) to such a pre-image, whereas mapping the elements $av$ for $v\in B_4$ to $0$. (We let it behave arbitrarily on a complement to the subspace that these elements span.)  Thus, (2) will hold.

Note that $ara$ annihilates elements of both $B_1$ and $B_3$, because these lie in $ker a$; and by the second condition of (2) it also annihilates elements of $B_4$;  while by the first condition of (2), it fixes elements of $B_2$. Hence

\medskip

\centerline{(3) ~ $ara$ is the projection $e$ of $V$ to the summand $V_2$ in the equality (1).}

\medskip

We now consider the element $a(1-e)$. By virtue of (3), this annihilates $V_2$, while acting on the other $V_i$ as $a$  does. By definition of the $V_i$'s, the element $a$ annihilates both $V_1$ and $V_3$, and sends $V_4$ into $im a$, the direct sum of $V_1$ and $V_2$, so we come to the conclusion that

\medskip

\centerline{(4) ~ $a(1-e)$ annihilates $V_1$, $V_2$, $V_3$ and sends $V_4$ into $V_1\oplus V_2$.}

\medskip

In particular, this shows that the image of $a(1-e)$ is contained in its kernel, thus proving the wanted equality $(a(1-e))^2=0$ after all.
\end{proof}

As an immediate consequence, we yield:

\begin{cor}\label{end} If $V_i$ ($i\in I$) are a family of vector spaces over fields $K_i$, and we let

$$
R=\prod_{i\in I} End_{K_i} V_i,
$$

\medskip

\noindent then, for every $a\in R$, there exists an idempotent $e\in aRa$ such that $(a(1-e))^2=0$.
\end{cor}

It was shown in \cite[Proposition 2.5]{D2} that regularly nil clean rings $R$ (and thus, in particular, D-regularly nil clean rings as well) are Utumi rings in the sense that, for every $x\in R$, there exists $y\in R$ depending on the element $x$ such that $x-x^2y \in Nil(R)$. It was also asked in the introductional section there are these Utumi's rings symmetric in the sense that $x-yx^2\in Nil(R)$. In what follows, we will give a positive solution to that question.

\begin{thm} The rings of Utumi are left-right symmetric.
\end{thm}

\begin{proof} Let $x\in R$ be an arbitrary element. Hence, by definition, there is $y\in R$ depending on $x$ such that $x-x^2y\in Nil(R)$. We claim that $x-yx^2\in Nil(R)$ which will show the desired symmetry. In fact, note firstly that for all $n\in \mathbb{N}$ one has that $(x-x^2y)^n=[x(1-xy)]^n=x(x-xyx)^{n-1}(1-xy)$. Thus, if $(x-x^2y)^n=0$, then one observes that $(x-xyx)^{n+1}=(1-xy)(x-x^2y)^nx=0$. Since we can analogously write that $(x-yx^2)^{n+2}=(1-yx)(x-xyx)^{n+1}x$, we elementarily see that this is zero too, i.e., $(x-yx^2)^{n+2}=0$, as required.
\end{proof}

Now, in order to support Definition~\ref{1}, we continue with a series of examples showing unambiguously the abundance and the complicated structure of D-regularly nil clean rings as well as that the converse implication in Theorem~\ref{strpireg} is not fulfilled.

\begin{ex}\label{cur1} There exists a D-regularly nil clean ring of index $2$, which is unit-regular, strongly clean and nil-clean, but is {\it not} strongly $\pi$-regular.
\end{ex}

\begin{proof} Consider the example of a nil-clean, unit-regular ring $R$ as constructed in \cite[Example 3.1]{S}, which ring is {\it not} strongly $\pi$-regular. In fact, it is of the form $R=\prod_{n=1}^{\infty} \mathbb{M}_n(\mathbb{Z}_2)$, and on p.345 from \cite{S} it was mentioned only that $R$ is regular, but this can be plainly strengthened to unit-regularity thus: Since $\mathbb{M}_n(\mathbb{Z}_2)$ is finite, and hence artinian, with zero Jacobson radical, it is both unit-regular and strongly $\pi$-regular (see \cite{E}). Now, it readily follows by coordinate-wise arguments that $R$ is also unit-regular, too, as claimed.

\medskip

Next, what we intend to show is that this ring is D-regularly nil clean, which will be substantiated by illustrating of two different ideas:

\medskip

\noindent{\bf First Approach.}  This follows at once from Theorem~\ref{reg}, as asserted.

\medskip

\noindent{\bf Second Approach.} In particular to Corollary~\ref{end}, this holds when the index set $I$ is the set of all natural numbers, and each $V_i=K^i=K\times \cdots \times K$ ($i$ times) for a common field $K=K_i$, so that we will have $R=\prod_{i\in I} \mathbb{M}_i(K)$, taking into account the classical facts that $End_K(K)\cong K$ and that $End_K(K^i)\cong \mathbb{M}_i(K)$. If we now just substitute $K=\mathbb{Z}_2$, the initial statement will be true, as desired, and the proof of the example is completed.
\end{proof}

\begin{rem}\label{new} We have proven even something more than we have claimed, namely, for any positive integer $n$ and any field $K$, the ring $\mathbb{M}_n(K)$ and its infinite direct product on $n$, where $n$ ranges over all naturals $\mathbb{N}$, are D-regularly nil clean of index $2$; when $n=0$ we just have at once that the index is $1$ for the field $K$. However, still left-open remain the question of whether or not $\mathbb{M}_n(R)$ is D-regularly nil clean for any $n\geq 1$, whenever $R$ is D-regularly nil clean (the reverse was obtained in Proposition~\ref{corner}).

In that aspect, combining a series of principally known facts, we may collect the following establishments (see, for more information, also \cite{R}): 

\medskip

$\bullet$ ~ it was proved in \cite{Lam} that, for any $n\in \mathbb{N}$ and any ring $R$, the ring $\mathbb{M}_n(R)$ is regular if, and only if, $R$ is regular.

\medskip

$\bullet$ ~ it was proved in \cite{He} that, for any $n\in \mathbb{N}$ and any ring $R$, the ring $\mathbb{M}_n(R)$ is unit-regular if, and only if, $R$ is unit-regular.

\medskip

$\bullet$ ~ it was proved in \cite{BDD} that, for any $n\in \mathbb{N}$ and any commutative ring $R$, the ring $\mathbb{M}_n(R)$ is strongly $\pi$-regular if, and only if, $R$ is strongly $\pi$-regular.

\medskip

It is still unknown whether these equivalencies hold for noncommutative (strongly) $\pi$-regular rings. The situation with matrix rings over noncommutative $\pi$-regular rings seems to be rather complicated. 

\medskip

$\bullet$ ~ it was proved in \cite{N2} that, for any $n\in \mathbb{N}$ and any ring $R$, the ring $\mathbb{M}_n(R)$ is exchange if, and only if, $R$ is exchange.

\medskip

$\bullet$ ~ it was proved in \cite{HN} that, for any $n\in \mathbb{N}$ and any ring $R$, the ring $\mathbb{M}_n(R)$ is clean if, and only if, $R$ is clean.

\medskip

For the case of strongly clean rings, this equivalence is no longer fulfilled, however (see, for consultation, \cite{BDD}).
\end{rem}

The last construction in the Example~\ref{cur1} above can be improved a bit more to the following one:

\begin{ex}\label{cur2} There exists a D-regularly nil clean ring of index at most $4$, which is both nil-clean and strongly clean, but is {\it not} $\pi$-regular.
\end{ex}

\begin{proof} We put into consideration the ring $R=\prod_{n=1}^{\infty} \mathbb{M}_n(\mathbb{Z}_4)$ from \cite[Example 3.2]{S}. It was established there that it is nil-clean but not $\pi$-regular. According to Remark~\ref{new}, $\mathbb{M}_n(\mathbb{Z}_4)$ is strongly $\pi$-regular as so is $\mathbb{Z}_4$. Consequently, $R$ is strongly clean (see, e.g., \cite{N3}).

\medskip

Further, what needs to prove is that $R$ is D-regularly nil clean, which we will show by demonstrating two independent methods: 

\medskip

\noindent{\bf First Approach.} As showed in Example 3.2 of \cite{S} there is a nil-ideal $I\subseteq J(R)$ such that $R/I$ is (unit-)regular being isomorphic to the ring $R=\prod_{n=1}^{\infty} \mathbb{M}_n(\mathbb{Z}_2)$ from the preceding Example~\ref{cur1}. We, therefore, apply Lemma~\ref{reduct} to get our pursued claim.

\medskip

\noindent {\bf Second Approach.} To that goal, by utilizing the next trick, we shall reduce the required proof to that of the previous example. Specifically, the following preliminary technical claim is true:

\begin{lem}\label{cr} Let $A$ be a commutative ring containing an element $\epsilon$ such that $\epsilon^2=0$, and let $R$ be an $A$-algebra containing an element $e$ which is an "approximate idempotent" in the sense that the following equality is fulfilled:

$$
(*) ~  e^2-e=\epsilon b, ~ \exists b\in R.
$$

\medskip

Then the element $e'=e-2e(\epsilon b)+\epsilon b$ is, in fact, an idempotent in $R$.
\end{lem}

\begin{proof} Note that under the presence of (*), the element $\epsilon b$ will commute with $e$, which can be routinely verified by a direct check. Hence the condition (*) and the equality above for $e'$ can be regarded as equations in the commutative $A$-subalgebra of $R$, generated by $e$ and $\epsilon b$. The calculation showing that $e'^2-e'=0$, based on the truthfulness of the equalities stated above, is now really straightforward.
\end{proof}

Now, we manage to prove the whole example. So, given any commutative ring $A$ with a square-zero element $\epsilon$ such that the factor-ring $A/\epsilon A$ is a field (e.g., such as $\mathbb{Z}_4$ with $\epsilon=2$), if we let $R$ be the endomorphism algebra of a free $A$-module $M$, then $R/\epsilon R$ will surely be the endomorphism ring of the ($A/\epsilon$)-vector space $M/\epsilon M$, and so given $a\in R$, we can successfully apply the internal Lemma~\ref{vs} to the image of $a$ in $R/\epsilon R$ to get an appropriate idempotent there.  Therefore, Lemma~\ref{cr} allows us to approximate this by an idempotent $e$ of $R$, which will approximately have the desired property, namely that the square $(a(1-e))^2$ will lie in $\epsilon R$. This will give that $(a(1-e))^4=0$, as required, completing the proof of the example.
\end{proof}

\section{Concluding Discussion and Open Questions}

In closing, we state into consideration the following five questions of some interest and importance to the subject:

\medskip

It is pretty easy to infer that finite direct products of (strongly) $\pi$-regular rings are (D-)regularly nil clean (it cannot be expected that these finite direct products will be again (strongly) $\pi$-regular even in the case of two direct components -- see, e.g., \cite{T}). However, in regard to Examples~\ref{cur1} and \ref{cur2}, one may ask the following:

\begin{prob} Does it follow that an infinite direct product of (strongly) $\pi$-regular rings is a (D-)regularly nil clean ring?
\end{prob}

The solution to the next three queries will increase all the picture about the relationships between the classes of rings examined above.

\begin{prob} Does there exist a D-regularly nil clean ring which is {\it not} strongly clean (and even not clean)?
\end{prob}

In regard to Theorem~\ref{reg}, one may state the following query:

\begin{prob}\label{15} Does there exist a $\pi$-regular ring which is {\it not} D-regularly nil clean?
\end{prob}


Perhaps does there exist a $\pi$-regular ring $R$ in which there is a regular element $a\in R$ such that $a^k$ is not regular for any $k>1$. Such an element, if it eventually exists, should not be unit, nor idempotent, nor nilpotent. On the same vein, such a ring, if it eventually exists, must have an unbounded index of nilpotence (for otherwise, it should be strongly $\pi$-regular by virtue of \cite{A} and so D-regularly nil clean according to Theorem~\ref{strpireg}).

However, at present, such a construction is not known to the author.

\medskip

Strengthening \cite[Problem 3.4]{D2}, one may ask

\begin{prob} Does there exist a nil-clean ring which is {\it not} D-regularly nil clean?
\end{prob}

In \cite{K} was proved the remarkable fact that in regular rings all strongly $\pi$-regular elements (and, in particular, all nilpotent elements) are always unit-regular. In addition, if a ring is simultaneously regular and strongly $\pi$-regular, then it is unit-regular.

\medskip

So, as a final query, it is reasonably adequate to be searching for the following expansion of the last affirmation in case that Problem~\ref{15} holds in the affirmative:

\begin{prob} In $\pi$-regular rings are all D-regular nil clean elements strongly $\pi$-regular? In addition, is a D-regularly nil clean ring strongly $\pi$-regular, provided that it is $\pi$-regular?
\end{prob}

\noindent{\bf Acknowledgements.} The author would like to express his sincere appreciation to Professor George M. Bergman from Berkeley, California, to Professor Kenneth R. Goodearl from Santa Barbara, California, to Professor Zachary Mesyan from Colorado, Colorado Springs and to Dr. Nik Stopar from Ljubljana, Slovenia, for making up possible a valuable discussion on the present topic.

\end{document}